\documentclass[11pt]{amsart}%
\usepackage{hyperref}
\usepackage{amsmath}%
\usepackage{amsfonts}%
\usepackage{amssymb}%
\usepackage{amsthm}
\usepackage{mathrsfs}
\usepackage{comment}
\usepackage{todonotes}
\usepackage{bbm}
\usepackage{amscd}
\usepackage{youngtab}
\usepackage{tikz}
\usepackage{bm}
\usepackage{enumitem}
\usetikzlibrary{arrows}
\usetikzlibrary{matrix}

\newcommand{\R}{\mathbb{R}}

\newcommand{\C}{\mathbb{C}}

\newcommand{\mf}{\mathfrak}

\DeclareMathOperator{\sgn}{sgn}

\newtheorem{theorem}{Theorem}[section]
\newtheorem{def-prop}[theorem]{Definition-Proposition}
\newtheorem{prop}[theorem]{Proposition}
\newtheorem{conj}[theorem]{Conjecture}
\newtheorem{lemma}[theorem]{Lemma}

\theoremstyle{definition}

\theoremstyle{remark}
\newtheorem*{remark}{Remark}

\begin{document}

\title[The Sperner property for the weak order]{A combinatorial $\mathfrak{sl}_2$-action and the Sperner property for the weak order}
\author{Christian Gaetz}
\thanks{C.G. is partially supported by an NSF Graduate Research Fellowship.}
\author{Yibo Gao}
\address{Department of Mathematics, Massachusetts Institute of Technology, Cambridge, MA.}
\email{\href{mailto:gaetz@mit.edu}{gaetz@mit.edu}} 
\email{\href{mailto:gaoyibo@mit.edu}{gaoyibo@mit.edu}}
\date{\today}

\begin{abstract}
We construct a simple combinatorially-defined representation of $\mathfrak{sl}_2$ which respects the order structure of the weak order on the symmetric group.  This is used to prove that the weak order has the strong Sperner property, and is therefore a Peck poset, solving a problem raised by Bj\"orner (1984); a positive answer to this question had been conjectured by Stanley (2017). 
\end{abstract}

\maketitle
\addtocounter{footnote}{1}
\footnotetext{Subject classification (MSC 2010): 06A07, 06A11, 05E18.  Keywords: weak Bruhat order, Sperner, Peck.}
\addtocounter{footnote}{1}
\footnotetext{An extended abstract of this work will appear in the proceedings of FPSAC 2019 \cite{extended-abstract}.}

\section{Introduction} \label{sec:intro}

\subsection{The Sperner property} \label{sec:Sperner}
We refer the reader to \cite{Stanley2012} for basic facts and terminology about posets in what follows.

Let $P$ be a finite ranked poset with rank decomposition
\[
P=P_0 \sqcup \cdots \sqcup P_r.
\]
We say that $P$ is \emph{k-Sperner} if no union of $k$ antichains of $P$ is larger than the union of the largest $k$ ranks.  If $P$ is $k$-Sperner for $k=1,...,r$, we say that $P$ is \emph{strongly Sperner}.  Let $p_i=|P_i|$, then we say $P$ is \emph{rank symmetric} if $p_i=p_{r-i}$ and \emph{rank unimodal} if 
\[
p_0 \leq p_1 \leq \cdots \leq p_{j-1} \leq p_j \geq p_{j+1} \geq \cdots \geq p_r
\]
for some $j$.  If $P$ is rank-symmetric, rank-unimodal, and strongly Sperner, then $P$ is \emph{Peck}. 

The Sperner property has long been of interest in both extremal and algebraic combinatorics.  For example, Sperner's Theorem, which asserts that the Boolean lattice $B_n$ is Sperner, is central to extremal set theory.  In \cite{Stanley1980}, Stanley used the Hard Lefschetz Theorem from algebraic geometry to prove that a large class of posets are strongly Sperner and obtained the Erd\H{o}s-Moser Conjecture as a corollary.  This class includes the strong Bruhat order (see Section \ref{sec:weak-order}), but not the weak order considered here.

\subsection{The weak order} \label{sec:weak-order}
Let $S_n$ denote the symmetric group of permutations of $n$ elements, viewed as a Coxeter group with respect to the simple transpositions $s_i=(i \: i+1)$ for $i=1,...,n-1$.  The \emph{weak order} $W_n=(S_n, \leq)$ is the poset structure on $S_n$ whose cover relations are defined as follows: $u \lessdot w$ if and only if $w=us_i$ for some $i$ and $\ell(w)=\ell(u)+1$, where $\ell$ denotes Coxeter length.  This poset is ranked with the rank of $w$ given by $\ell(w)$; the unique permutation $w_0$ of maximum length has one-line notation $n \: (n-1) \: (n-2)\: ... \: 1$ and length ${n \choose 2}$.

This definition is in contrast to the \emph{strong order} (or \emph{Bruhat order}) on $S_n$ which has cover relations corresponding to right multiplication by any transposition $t_{ij}=(i \: j)$ (still subject to the condition that length increases by one), rather than just the simple transpositions $s_i$; it was proven in \cite{Stanley1980} that the strong order is Peck.  The weak and strong orders share the same ground set and rank structure, so the weak order is rank-symmetric and rank-unimodal.  The Sperner property of $W_n$, however, does not follow from that of the strong order, since $W_n$ has many fewer covering relations.  

Whether $W_n$ is Sperner has been investigated at least since Bj\"orner in 1984 \cite{Bjorner1984}, and a positive answer was conjectured by Stanley \cite{Stanley2017}.  Our main result is a positive answer to this problem:

\begin{theorem} \label{thm:main-theorem}
For all $n \geq 1$ the weak order $W_n$ is strongly Sperner, and therefore Peck.
\end{theorem}

\subsection{Order raising operators} \label{sec:raising}

For $P = P_0 \sqcup \cdots \sqcup P_r$ a finite ranked poset, and $S \subseteq P$, let $\C S$ denote the vector space of formal linear combinations of elements of $S$.  A linear map $U: \C P \to \C P$ sending elements $x \in P$ to $\sum_y c_{xy} y$ is called an \emph{order raising operator} if $c_{xy}=0$ unless $x \lessdot y$.  Any linear map $D: \C P \to \C P$ sending each subspace $\C P_k \to \C P_{k-1}$ is called a \emph{lowering operator}. 

\begin{prop}[Stanley \cite{Stanley1980}]
\label{prop:order-raising}
Suppose there exists an order raising operator $U: \C P \to \C P$ such that if $0 \leq k < \frac{r}{2}$ then $U^{r-2k}: \C P_k \to \C P_{r-k}$ is invertible.  Then $P$ is strongly Sperner.
\end{prop}

In \cite{Stanley2017}, Stanley suggested that the order raising operator $U: \C W_n \to \C W_n$ defined for $w \in W_n$ by
\[
U \cdot w = \sum_{i: \: \ell(ws_i)=\ell(w)+1} i \cdot ws_i
\]
and extended by linearity may have the desired property.  He conjectured an explicit non-vanishing product formula for the determinants of the maps $U^{{n \choose 2}-2k}: \C (W_n)_k \to \C (W_n)_{{n \choose 2} -k}$ for $0 \leq k < \frac{1}{2}{n \choose 2}$, which, by Proposition \ref{prop:order-raising} would imply Theorem \ref{thm:main-theorem}.

In Section \ref{sec:sl2-action}, we prove that $U^{{n \choose 2}-2k}$ is invertible by constructing a representation of $\mathfrak{sl}_2$ on $\C W_n$ with weight spaces $\C (W_n)_i$ such that the standard generator $e \in \mathfrak{sl}_2$ acts by $U$ (a result of Proctor \cite{Proctor1982} implies that, if $W_n$ is Peck, then there is \emph{some} such representation in which $e$ acts as an order raising operator). 

\begin{remark}
In subsequent work, Hamaker, Pechenik, Speyer, and Weigandt have interpreted this $\mf{sl}_2$-action in terms of derivatives of Schubert polynomials, allowing them to prove Stanley's conjectured determinant \cite{Hamaker2018}.
\end{remark}

\section{An action of $\mathfrak{sl}_2$} \label{sec:sl2-action}
We define a lowering operator $D:\C W_n\rightarrow \C W_n$ by
$$D\cdot w=\sum_{\substack{1 \leq i < j \leq n \\ \ell(wt_{ij})=\ell(w)-1}} c(w,wt_{ij}) \cdot wt_{ij}$$
with weights $c(w,wt_{ij}) = 2\big(w_i-w_j-a(w,wt_{ij})\big)-1$
where $a(w,wt_{ij}):=\#\{k<i:w_j<w_k<w_i\}.$  Note that the sum in our definition is over all covering relations in the \emph{strong} Bruhat order (the fact that $D$ does not respect the \emph{weak} order will be immaterial to our argument; see \cite{Gaetz2018} for an investigation of the combinatorial implications for the weak and strong orders).  

An alternative description of the weights $c(w,wt_{ij})$ can also be given in terms of Lehmer codes.  The \emph{Lehmer code} $L(w)$ of a permutation $w \in S_n$ is the tuple $(L_1,...,L_n)$ with $L_i:=\# \{j>i \: | \: w_j < w_i\}$; the map $w \mapsto L(w)$ is well known to be a bijection $S_n \to [0,n-1] \times [0,n-2] \times \cdots \times [0,1] \times \{0\}$ (see, for example, Stanley \cite{Stanley2012}, where $L(w)$ is exactly the \emph{inversion table} for $w^{-1}$).  Under the condition $\ell(wt_{ij})=\ell(w)-1$ it is not hard to see that 
\[
c(w,wt_{ij})=\left( L_i(w)-L_i(wt_{ij}) \right) + \left( L_j(wt_{ij})-L_j(w) \right).
\]  
That is, $c(w,wt_{ij})$ is the Manhattan (or $L^1$) distance between the tuples $L(w)$ and $L(wt_{ij})$.

See Figure \ref{fig:labels3} for a depiction of the order raising operator $U$ and the lowering operator $D$ in the case $n=3$.

\begin{figure}
\begin{center}
\begin{tikzpicture} [scale=1.6]
\node[draw,shape=circle,fill=black,scale=0.8] (a0)[label=below:{$123$}] at (0,0) {};
\node[draw,shape=circle,fill=black,scale=0.8](b0)[label=left:{$213$}] at (-1,1) {};
\node[draw,shape=circle,fill=black,scale=0.8](b1)[label=right:{$132$}] at (1,1) {};
\node[draw,shape=circle,fill=black,scale=0.8](c0)[label=left:{$231$}] at (-1,2) {};
\node[draw,shape=circle,fill=black,scale=0.8](c1)[label=right:{$312$}] at (1,2) {};
\node[draw,shape=circle,fill=black,scale=0.8](d0)[label={$321$}] at (0,3) {};

\draw (a0) -- node[left] {$1$} (b0);
\draw (a0) -- node[right] {$2$} (b1);
\draw (b0) -- node[left] {$2$} (c0);
\draw (b1) -- node[right] {$1$} (c1);
\draw (c0) -- node[left] {$1$} (d0);
\draw (c1) -- node[right] {$2$} (d0);
\end{tikzpicture} \hspace{0.2in}
\begin{tikzpicture} [scale=1.6]
\node[draw,shape=circle,fill=black,scale=0.8] (a0)[label=below:{$123$}] at (0,0) {};
\node[draw,shape=circle,fill=black,scale=0.8](b0)[label=left:{$213$}] at (-1,1) {};
\node[draw,shape=circle,fill=black,scale=0.8](b1)[label=right:{$132$}] at (1,1) {};
\node[draw,shape=circle,fill=black,scale=0.8](c0)[label=left:{$231$}] at (-1,2) {};
\node[draw,shape=circle,fill=black,scale=0.8](c1)[label=right:{$312$}] at (1,2) {};
\node[draw,shape=circle,fill=black,scale=0.8](d0)[label={$321$}] at (0,3) {};

\draw (a0) -- node[left] {$1$} (b0);
\draw (a0) -- node[right] {$1$} (b1);
\draw (b0) -- node[left] {$1$} (c0);
\draw (b1) -- node[right] {$3$} (c1);
\draw (b0) -- node[below] {$1$} (c1);
\draw (b1) -- node[above] {$1$} (c0);
\draw (c0) -- node[left] {$1$} (d0);
\draw (c1) -- node[right] {$1$} (d0);
\end{tikzpicture}
\end{center}
\caption{The edge weights for the order raising operator $U$ (left) and the lowering operator $D$ (right).}  
\label{fig:labels3}
\end{figure}
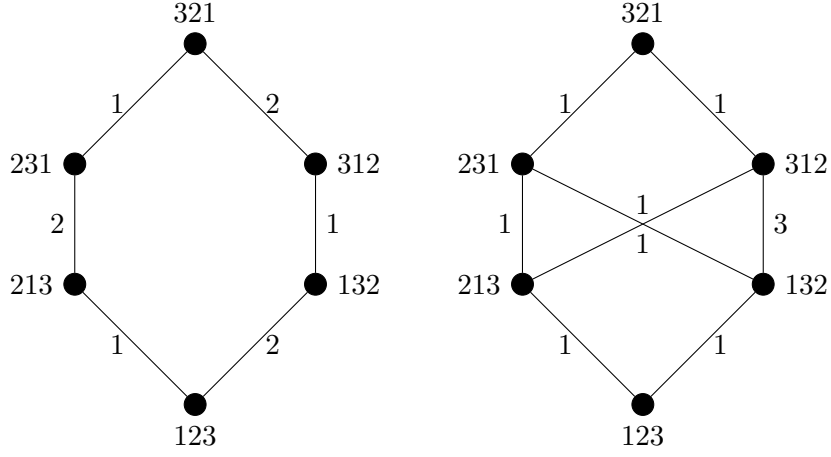

We also define a modified rank function $H: \C W_n \to \C W_n$ by \[
H(w)=\left(2 \cdot \ell(w) - {n \choose 2}\right) \cdot w
\]
for $w \in W_n$ and extending by linearity; this choice is necessiated by the weight theory of $\mf{sl}_2$-representations.  Since $H$ acts as a multiple of the identity on each rank, it can be shown (see Proctor \cite{Proctor1982}) that for \emph{any} raising operator $U$ and lowering operator $D$ we have
\begin{align} \label{eq:weight-relations1}
    HU-UH&=2U  \\
    \label{eq:weight-relations2}
    HD-DH&=-2D.
\end{align}

In this section, we show that $U,D$ together with $H$ provide a representation of $\mathfrak{sl}_2$ on $\C W_n$.  The Lie algebra $\mathfrak{sl}_2(\C)$ has a standard linear basis
\begin{align*}
    e &= \begin{pmatrix} 0 & 1 \\ 0 & 0 \end{pmatrix} \\
    f &= \begin{pmatrix} 0 & 0 \\ 1 & 0 \end{pmatrix} \\
    h &= \begin{pmatrix} 1 & 0 \\ 0 & -1 \end{pmatrix}
\end{align*}
and is determined by the relations $[h,e]=2e$, $[h,f]=-2f$, and $[e,f]=h$.  Here $[,]$ denotes the standard Lie bracket: $[X,Y]:=XY-YX$.

In light of (\ref{eq:weight-relations1}) and (\ref{eq:weight-relations2}), all that remains is to check that $[U,D]=H$.  We can view $[U,D]=UD-DU$ and $H$ as matrices of size $n! \times n!$ with rows and columns indexed by permutations, and we show that they are equal by comparing entries via Lemma~\ref{lem:self} and Lemma~\ref{lem:diamond} below.
\begin{lemma}\label{lem:self}
For every $w\in W_n$, $(UD-DU)_{w,w}=2 \cdot \ell(w)-{n\choose 2}.$
\end{lemma}
\begin{proof}
Assume $w\in (W_n)_k$, meaning $\ell(w)=k$. We have that, by definition,
\begin{align*}
UD_{w,w}=&\sum_{u\in (W_n)_{k-1}}D_{u,w}\cdot U_{w,u}=\sum_{u\lessdot_{W_n}w}D_{u,w}\cdot U_{w,u}\\
=&\sum_{i:\ w_i>w_{i+1}}i\cdot\Big(2\big(w_i-w_{i+1}-a(w,ws_i)\big)-1\Big)\\
=&\sum_{i:\ w_i>w_{i+1}}2i(w_i-w_{i+1}) \\& +\sum_{i:\ w_i>w_{i+1}}\big(-2i\cdot\#\{j<i:w_{i+1}<w_j<w_i\}-i\big),
\end{align*}
where $\lessdot_{W_n}$ denotes the covering relations in the weak order $W_n$.
Similarly,
\begin{align*}
-DU_{w,w}=&\sum_{i:\ w_i<w_{i+1}}2i(w_i-w_{i+1}) \\ &+\sum_{i:\ w_i<w_{i+1}}\Big(2i\cdot\#\{j<i:w_i<w_j<w_{i+1}\}+i\Big).
\end{align*}
Putting them together, we obtain
\begin{align*}
(UD-DU)_{w,w}=&\sum_i 2i(w_i-w_{i+1})+A\\
=&2(w_1-w_2)+\cdots+(2n-2)(w_{n-1}-w_n)+A\\
=&n^2+n-2nw_n+A,
\end{align*}
where by switching the order of summation,  we can write $A$ as a sum over $j$'s instead of $i$'s as above:
\[
A=\sum_{j=1}^{n-1}\left(\sgn(w_{j+1}-w_j)\cdot j-2 \left( \sum_{\substack{i:\ i>j\\w_{i+1}<w_j<w_i}}i\right)+2 \left( \sum_{\substack{i:\ i>j\\w_{i}<w_j<w_{i+1}}}i\right)\right).
\]
Here $\sgn:\R^{\times}\rightarrow\{\pm1\}$ is the sign function. Let $A_j$ denote the above summand, so that $A=\sum_{j=1}^{n-1}A_j$.

Assume first that $w_j<w_{j+1}$ so that $\sgn(w_{j+1}-w_j)=1$. Let $j<j_1<j_2<\cdots<j_p$ be all the indices such that $w_{j_m}-w_j$ and $w_{j_m+1}-w_j$ have different signs, for $m\geq1$. Since we have assumed that $w_j<w_{j+1}$, we know $w_j<w_{j_1}$, $w_j>w_{j_1+1}$, $w_j>w_{j_2}$, $w_{j}<w_{j_2+1}$ and so on. As a result,
\begin{align*}
A_j=&j-2(j_1+j_3+\cdots)+2(j_2+j_4+\cdots)\\
=&-(j_1-j)+(j_2-j_1)-(j_3-j_2)+\cdots\pm j_p\\
=&-(j_1-j)+(j_2-j_1)-(j_3-j_2)+\cdots\pm (n-j_p)\pm n\\
=&\#\{i>j:w_i<w_j\}-\#\{i>j:w_i>w_j\}\pm n,
\end{align*}
where the last sign is $+$ if $w_j<w_n$ and is $-$ if $w_j>w_n$. The case $w_j>w_{j+1}$ yields the exact same formula with the same argument.

Once we consider all the $A_j$'s together, the last terms $\pm n$ will appear as $+n$ for $w_n-1$ times and will appear as $-n$ for $n-w_n$ times. Therefore,
\begin{align*}
(UD-DU)_{w,w}=&n^2+n-2nw_n+\sum_{j=1}^{n-1}A_j\\
=&n^2+n-2nw_n+\#\{i>j:w_i<w_j\}\\
&-\#\{i>j:w_i>w_j\}+n(w_n-1)-n(n-w_n)\\
=&k-\left({n\choose2}-k\right)=2k-{n\choose 2}.
\end{align*}
\end{proof}

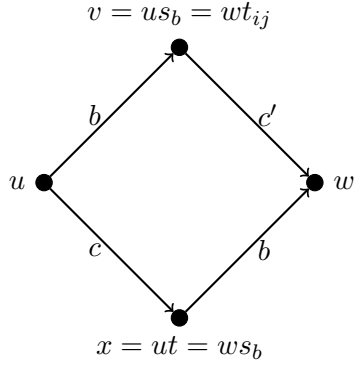
\begin{figure}
\begin{center}
\begin{tikzpicture} [scale=1.8]
\node[draw,shape=circle,fill=black,scale=0.6] (c0)[label=above:{$v=us_b=wt_{ij}$}] at (0,2) {};
\node[draw,shape=circle,fill=black,scale=0.6](b0)[label=left:{$u$}] at (-1,1) {};
\node[draw,shape=circle,fill=black,scale=0.6](b1)[label=right:{$w$}] at (1,1) {};
\node[draw,shape=circle,fill=black,scale=0.6](a0)[label=below:{$x=ut=ws_ b$}] at (0,0) {};

\draw[->] (b0) edge[thick] node[left] {$c$} (a0);
\draw[->] (a0) edge[thick] node[right] {$b$} (b1);
\draw[->] (b0) edge[thick] node[left] {$b$} (c0);
\draw[->] (c0) edge[thick] node[right] {$c'$} (b1);
\end{tikzpicture} 
\end{center}
\caption{The paths under consideration in the proof of Lemma \ref{lem:diamond}.  The transposition $t$ varies from case to case, however the weights $c$ and $c'$ are always equal.}  
\label{fig:diamond-lemma}
\end{figure}

\begin{lemma}\label{lem:diamond}
For $w\neq u\in W_n,(UD-DU)_{w,u}=0.$
\end{lemma}
\begin{proof}
It suffices to check cases where $(UD)_{w,u}\neq0$ or $(DU)_{w,u}\neq0$. Let $k=\ell(w)=\ell(u)$. Suppose that $(DU)_{w,u}\neq0$, in which case there exists $v=us_b$ with $\ell(v)=\ell(u)+1$ and $v=wt_{ij}$ $(i<j)$ with $\ell(v)=\ell(w)+1$. We view $W_n$ as a directed graph with up edges corresponding to covering relations of the weak Bruhat order $W_n$ and down edges corresponding to covering relations in the strong Bruhat order (see Figure~\ref{fig:diamond-lemma}). Several cases are considered below; in each case there are exactly two directed paths of length two from $u$ to $w$: one going up-down and the other down-up. The edge weights of these two paths always agree and thus give $(UD-DU)_{w,u}=0$.

\textbf{Case 1:} $\{b,b+1\}\cap\{i,j\}=\emptyset$. It is clear that there are exactly two directed paths from $u$ to $w$ of length 2, which are $u\rightarrow v\rightarrow w$ and $u\rightarrow x=ut_{ij}\rightarrow w$. By definition, $U_{v,u}=U_{w,ut_{ij}}=b$ and $D_{w,v}=D_{ut_{ij},u}$.

\textbf{Case 2:} $b=i$. As $w\neq u$, we must have $j>b+1$. By our condition on the path $u\rightarrow v\rightarrow w$, we know that $u_b<u_j<u_{b+1}$ and therefore there exists one more path of length 2 from $u$ to $w$, which is $u\rightarrow x\rightarrow w$ where $x=ut_{i+1,j}=ws_b$. The up edges of these two paths both have weight $b$ and for the down edges, $D_{w,v}=2(w_j-w_i-a(v,w))-1$ and $D_{x,u}=2(u_{i+1}-u_j-a(u,x))-1$. We have $w_j=u_{i+1}$ and $w_i=u_j$ and since $u_i$ is less than both $u_{i+1}$ and $u_j$, we conclude that $a(v,w)=a(u,x)$.

\textbf{Case 3:} $b+1=i$. Similarly, we know that $u_j<u_b<u_{i}=u_{b+1}$ and the only other directed path is $u\rightarrow x\rightarrow w$ with $x=ut_{b,j}=ws_b$. The up edges both have weight $b$; for the down edges the key parameters $a(v,w)$ and $a(u,x)$ are equal, since $u_{b+1}$ is greater than both $u_b$ and $u_j$.

\textbf{Case 4:} $b=j$. Here $u_b=u_j<u_{b+1}<u_i$ and the other path is $u\rightarrow x\rightarrow w$ with $x=ut_{i,j+1}=ws_b$. The up edges have the same weight $b$ and the down edges have the same weight since the transpositions $w=vt_{i,j}$ and $x=ut_{i,j+1}$ both swap entries with the same values, and all four permutations have the same values at indices $1,2,\ldots,i-1$.

\textbf{Case 5:} $b+1=j$. As $w\neq u$, we know $i<b$. First, $u_b<u_{b+1}$ and since $w=vt_{ij}$ with $\ell(w)=\ell(v)-1$, we must have $u_b<u_i<u_j=u_{b+1}$. The other path is $u\rightarrow x\rightarrow w$ with $x=ut_{i,j-1}=ws_b$. The up edges have the same weight $b$ and the down edges have the same weight since the transpositions $w=vt_{i,j}$ and $x=ut_{i,j-1}$ both swap entries with the same values, and all four permutations have the same values at indices $1,2,\ldots,i-1$.

Beginning instead with the assumption that  $(UD)_{w,u}\neq 0$, it is easy to see that all cases are already included above.
\end{proof}

We now complete the proof of the main theorem, which follows from Theorem 1 of Proctor \cite{Proctor1982}, given Lemmas \ref{lem:self} and \ref{lem:diamond}. We include the proof here because it is so simple.

\begin{proof}[Proof of Theorem \ref{thm:main-theorem}]
Lemma \ref{lem:self} and Lemma \ref{lem:diamond} together show that the map sending $e \mapsto U$, $f \mapsto D$, and $h \mapsto H$ defines a representation of $\mathfrak{sl}_2(\C)$ on $\C W_n$ with weight spaces $\C (W_n)_k$ of weight $2k-{n \choose 2}$.  It is an immediate consequence of the theory of highest weight representations (see, for example, Theorem 4.60 of \cite{Kirillov2008}) that $U^{{n \choose 2}-2k}: \C (W_n)_k \to \C (W_n)_{{n \choose 2}-k}$ is an isomorphism.  Since $U$ is an order raising operator by definition, Proposition \ref{prop:order-raising} implies the desired result.
\end{proof}

\section{Other Coxeter types} \label{sec:other-types}

The weak and strong Bruhat orders generalize naturally to any finite Coxeter group $C$, with the role of the simple transpositions $(i \: i+1)$ replaced by any choice of simple reflections, and the set of all transpositions $(i \: j)$ replaced by the set of all reflections in $C$.  Stanley's result \cite{Stanley1980} that the strong order is strongly Sperner applies to any finite Weyl group.  An easy argument proves the same for the dihedral groups, and computer checks verify that the strong orders on the exceptional Coxeter groups of types $H_3$ and $H_4$ are also strongly Sperner.  Since strong orders for all Coxeter types are known to be rank-symmetric and rank-unimodal, and since products of Peck posets are known to be Peck \cite{Proctor1980}, it follows that Stanley's result can be extended to all finite Coxeter groups.  Our results for the weak order apply only to the symmetric group, however we conjecture that they can be extended to all finite Coxeter groups:

\begin{conj}
The weak order on any finite Coxeter group strongly Sperner.
\end{conj}

An easy argument proves the Conjecture for the dihedral groups, and computer checks have also verified it for all Coxeter groups of rank at most four.

\section*{Acknowledgements}
The authors wish to thank Richard Stanley for teaching them about posets, and for helpful discussions about this problem.

\bibliographystyle{plain}
\bibliography{arxiv-v2}
\end{document}